\documentclass[sigconf]{acmart}

\usepackage{latexsym,amsxtra}
\usepackage{amsfonts}
\usepackage{verbatim}
\usepackage{amsmath}
\usepackage{amssymb}

\usepackage{balance}

\numberwithin{equation}{section}

\theoremstyle{plain}
\newtheorem{cor}[theorem]{Corollary}

\theoremstyle{definition}
\newtheorem{defi}[theorem]{Definition}

\newtheorem{rema}[theorem]{Remark}
\newcommand{\KK}{\ensuremath{\mathbf k}}  % basic field

\newcommand{\s}{\ensuremath{\mathbb S}}
\newcommand{\Z}{\ensuremath{\mathbb Z}}

\newcommand\alia{ {\mathcal A}lia } % a notation for alia operad

\def\dim{\mbox{dim\,}}
\newcommand \f{\mathcal{F}}         % for free operad

\def\Id{\mbox{Id\,}}
\def\tor{\mbox{Tor\,}}

\def\ass{\: {\mathcal A}ss \:}

\def\p{ {\mathcal P} }

\copyrightyear{2017}
\acmYear{2017}
\setcopyright{acmcopyright}
\acmConference{ISSAC '17}{July 25-28, 2017}{Kaiserslautern, Germany}\acmPrice{15.00}\acmDOI{http://dx.doi.org/10.1145/3087604.3087655}
\acmISBN{978-1-4503-5064-8/17/07}

\begin{document}

\title {Growth in Varieties of Multioperator Algebras 
            \\and Groebner Bases in Operads}
%\title{Operads vs Identities:\\
%a dictionary of linear universal algebra}
%\thanks{Partially
%supported by the grant 05-01-01034 of the Russian Basic Research Foundation.}%

\author{Dmitri Piontkovski}

\authornote{The article was prepared within the framework of the Academic Fund Program at the National Research University Higher School of Economics (HSE) in 2017--2018 (grant 17-01-0006) and by the Russian Academic Excellence Project ``5--100''.}
\orcid{0000-0002-9853-1891}
\affiliation{%
  \institution{Department of Mathematics for Economics, National Research University Higher School of Economics}
  \streetaddress{Myasnitskaya str. 20}
  \city{Moscow} 
   \country{Russia} 
  \postcode{101000}
}
\email{dpiontkovski@hse.ru}

\date{\today}

 \begin{abstract}
We consider varieties of linear multioperator algebras, that is, classes of algebras with several multilinear operations satisfying certain identities. To each such a variety one can assign a numerical sequence called a sequence of codimensions. The $n$-th codimension is equal to the dimension of the vector space of all $n$-linear operations in the free algebra of the variety.  In recent decades, a new approach to such a sequence has appeared based on the fact that the union of the above vector spaces carries the structure of algebraic operad, so that the generating function of the codimension sequence is equal to the generating series of the operad. 

We show that in general there does not exist an algorithm to decide whether  the growth exponent of the codimension sequence  of the variety defined by given finite sets of operations and identities is equal to a given rational number. In particular, we solve negatively a recent conjecture by Bremner and Dotsenko by showing that the set of codimension sequences of varieties defined by a bounded number and degrees of operations and identities is infinite. Then we discuss algorithms which in many cases calculate the generating functions of the codimension series in the form of a defining algebraic or differential equation. For a more general class of varieties, these algorithms give upper and lower bounds for the codimensions  in terms of generating functions.  The upper bound is just a formal power series satisfying an algebraic equation defined effectively by the generators and the identities of the variety.  The first stage of an algorithm for the lower bound  is the construction of a Groebner basis of the operad. If the Groebner basis happens to be finite and satisfies mild restrictions, a recent theorem by the author and Anton Khoroshkin guarantees that the desired generating function is either algebraic or differential algebraic.   We describe algorithms producing such equations. In the case of infinite Groebner basis, these algorithms applied to its finite subsets give lower bounds for the generating function of the codimension sequence. 
 \end{abstract}


%\ccsdesc[100]{Mathematics of computing~Enumeration}
%\ccsdesc[100]{Mathematics of computing~Differential algebraic equations}


\maketitle

%\section{A dictionary of universal algebra: \\
%operads vs identities}

\section{Introduction}

We consider classes of linear multioperator algebras defined by operations and identities among them, that is, varieties of multioperator algebras. Given such a variety, a sequence of vector spaces is associated to it called a 
{\it cocharacter sequence} or an {\it algebraic operad} associated to the variety.  The sequence of dimensions 
of these vector spaces ({\it codimension sequence}) and its (exponential) generating function ({\em codimension series} of the variety, or {\em generating series} of the operad) is one of the most important invariant of an operad or a variety. In particular, the asymptotic growth   of the cocharacter sequence is a measure of the growth of the operad. Both the generating series and the asymptotics of the coefficients have been studied for a number of varieties and operads, see~\cite{giza} and~\cite{zin} and references therein.

Here we discuss the algorithmic approach to determining asymptotic growth of  operads. First, we should mention the recent progress in determining such a growth in the case of varieties of algebras with one binary operation~\cite{giza} where the theory of representations of symmetric groups is extremely useful. In particular it is shown~\cite{berele} that the codimension series of each variety of associative algebras is a holonomic function.  In contrast, here we  concentrate on the case of the varieties of multi-operational algebras. In this case, operadic methods are extremely useful. These operadic methods lead to algorithms based on various symbolic computation concepts such as; the  theory of Groebner bases in operads, the combinatorial problems which are close to the problem of enumeration  of labelled trees avoiding certain patterns,  the formal power series solutions of algebraic and algebraic differential equations and the problem of determining  the asymptotics of the coefficients of such solutions.

 We try to explain here a  chain of algorithms which leads from a list of operations and identities to  either asymptotic growth of  corresponding variety or  bounds for  asymptotics.  We discuss  also theoretical obstructions for such algorithms to be applicable in the most general case. We  discuss briefly a recent implementation of some algorithms from the chain (that is,  algorithms for calculation of  Groebner bases in operads). 
 
 More precisely, we first show that, in general, there does not exist an algorithm which always  determines the growth asymptotics of codimensions of a variety defined by finite collection of operations and identities. Next, we give a lower bound for the generating series of the codimension sequence in  cases when the number of  defining identities is bounded in each arity. This lower bound is a formal power series solution of some algebraic equation whose coefficients encode  dimensions of  spaces of  generators and  relations of the operad.  In particular, this gives explicit lower bounds for the codimensions. The bound becomes an equality (that is, it gives a formula for the generating function of the operad) if the operad satisfies a simple homological condition, that is, it has right homological dimension at most two. 
 
  Then, we give  upper bounds for the generating functions of the codimension sequences of arbitrary operads. This upper bound is equal to the generating series of the monomial operad corresponding to any partial Groebner basis of a given operad. Under some mild conditions on  defining monomial relations of the monomial operad, this lower bound is the formal power series solution of a differential algebraic equation or a pure algebraic equation. So, this bound occurs to be an equality if the operad under consideration has finite Groebner basis.  Note that there are  operads for which  both of our bounds become equalities. 
  
  Consider, for example, the variety of alia algebras introduced by Dzhumadildaev, that is, the variety of non-associative algebras that are satisfying the identity
$$
\{[x_1, x_2], x_3\} + \{[x_2, x_3], x_1\} + \{[x_3, x_1], x_2\} = 0,
$$
where $[a,b] = ab-ba$ and $\{a,b\} = ab+ba$, these algebras are also referred to as 1-alia algebras~\cite{dzh}. 
It is shown in~\cite[Example~3.5.1]{kp} that the corresponding  operad $\alia$ has quadratic Groebner basis (in particular, it is Koszul). 
Its generating series $y=\alia(z)$ is equal to the  generating series 
of the corresponding quadratic monomial operad. 
Moreover, the relations of the monomial  operad are symmetric regular (see Section~\ref{sec:we}), so that $y$ satisfies 
an algebraic equation $y-y^2+y^3/6=z$. 
On the other hand, it follows that the generating series of the quadratic dual operad is $\alia^!(z) = z+z^2 +z^3/6$. It follows that the operad $\alia$ has a homological dimension of two, so that our lower bound becomes an equality too. Indeed, Corollary~\ref{cor:GS_quadr} implies that the  lower bound $y=y(z)$ satisfies the same algebraic equation.

The plan of the paper is as follows. In Section~\ref{sec1} we briefly recall what  varieties of algebras and algebraic operads are. In particular, we give a brief `phrase-book' for explaining operads in terms of varieties and vice versa.  In Section~\ref{sec:undec}, we explain a theoretical  obstruction for determining  growth of an operad. First, we show that for sufficiently large $n$ the set of generating functions of quadratic operads generated by $n$ binary operations (=  set of codimension series of varieties of algebras with $n$ binary operations) is infinite. This solves  a conjecture by Bremner and Dotsenko negatively~\cite[Conjecture 10.4.1.1]{db}. Moreover, we show that there does not exist an algorithm which, given a list of generators and relations of a non-symmetric operad  known to have an exponential growth, always decides whether the exponent of the growth is equal to a given rational number. This means that there is no algorithm to decide that, given a set of generators and identities of a variety of multioperator algebras such that its codimensions $c_n$ grow approximately as $ n! c^n $ for some $c\ge 1$,
whether $c$ is equal to a given rational number. So,
 the asymptotic growth of a variety defined by a finite number of identities is not algorithmically recognized in general.

In Section~\ref{sec:est} we give estimates for the dimensions of the components of operads provided that the number of defining identities in each degree is bounded. In this case, the generating series of an operad is bounded below by an algebraic function depending only on  degrees and  arities of  generators and  identities.
 Particularly, we use an operadic version of the Golod--Shafarevich theorem~\cite{kurosh} to establish
the lower and upper  asymptotic bounds of  type $c^n \frac{(2n)!}{n!} $ for quadratic operads generated by at least two non-symmetric operations with bounded number of defined relations in each degree. 

In Section~\ref{sec:groeb} we briefly recall the foundations of the theory of Groebner bases for operads~\cite{dk}. Then we discuss the operads with finite Groebner bases. The generating function of such an operad is equal to the generating function of the corresponding {\em monomial} operad. So, in the next Section~\ref{sec:we} we discuss  finitely presented monomial operads. We recall that under some mild symmetry conditions   the generating function of such operad does satisfy an algebraic differential equation. We also briefly describe an algorithm to generate such an equation based on the results of~\cite{kp}. We focus on the additional conditions that imply that the obtained equation is in fact algebraic or even rational. In the last two cases the asymptotics of the generating function coefficients can be recovered by the standard computer algebra tools.

\section{Operads and varieties}

\label{sec1}

For the details on operads we refer the reader to the monographs~\cite{oper} and~\cite{Loday_Valet}; see also the textbook~\cite{db}. 

\subsection{A definition of operad}

We consider multioperator linear algebras over a field $\KK$ of zero characteristic.
Let $W$ be a variety of $\KK$--linear algebras (without constants,
with identity and without other unary operations)  of some
signature $\Omega$. We assume  $\Omega$ is a finite union of finite sets $\Omega = \Omega_2 \cup \dots \cup \Omega_k$ where the 
elements $\omega$ of  $\Omega_t$ act on each algebra $A \in W$ as $t$-linear operations, $\omega:A^{\otimes t} \to A$. 
Recall that a variety is defined by two sets, the signature $\Omega$ and a set of defining identities $R$. By the linearization process, one can assume that $R$ consists of multilinear identities. 
Consider the free algebra $F^W (x)$ on a
countable set of indeterminates $x= \{ x_1, x_2, \dots \}$. Let
$\p_n \subset F$ be the subspace consisting of all multilinear
generalized homogeneous polynomials on the variables $\{ x_1,
\dots, x_n \}$, that is, $\p_n$ is the component $F^W (x) [1, \dots, 1, 0, 0,\dots]$ with respect to the $\Z^\infty$-grading by the degrees on $x_i$.

\begin{defi}
Given such a variety $W$,  the sequence $\p_W = \p := \{ \p_1,
\p_2, \dots \}$ of the vector subspaces of $F^W (x)$ is called
an {\it operad}\footnote{More precisely, symmetric connected
$\KK$--linear operad with identity.}. 
\end{defi}

The $n$-th component  $\p_n$  may be identified with the set of
all derived $n$-linear operations on the algebras of $W$; in
particular,  $\p_n$ carries a natural structure of a
representation of the symmetric group $S_n$. Such a sequence $Q =
\{ Q(n) \}_{n \in \Z}$ of representations $Q(n)$ of the symmetric
groups $S_n$ is called an {$\s$--module}, so that an operad
carries a structure of $\s$-module with $\p_n = \p(n)$.  Also, the compositions of
operations (that is, a substitution of an argument $x_i$ by a
result of another operation with a subsequent monotone re-numbering the inputs to avoid repetitions) gives natural maps of
$S_*$-modules $\circ_i :  \p(n)\otimes \p(m) \to \p(n+m-1)$. Note
that the axiomatization of these operations gives an abstract
definition of operads, see~\cite{oper} for the discussion.

Note that the signature $\Omega$ can be considered as a sequence of subsets of $\p$ with $\Omega_n \subset \p_n$. 
Then $\Omega$ generates the operad $\p$ up to the $\s$--module structure and the  compositions $\circ_i$ so that it is called a 
{\em set of generators} of the operad.  

More generally, the  $\s$-module $X$ 
generated by $\Omega$ is called the (minimal) {\em module of generators} of the operad $\p$.
It can be also defined independently of $\Omega$ as $X = \p_+/(\p_+\circ \p_+)$ where $\p_+ = p_2\cup p_3 \cup \dots$ and $\circ$
denotes the span of all compositions of two $\s$-modules.
Then one can define a variety $W$ corresponding to a (formal) operad $\p$ by picking a set $\Omega$
of generators of $X$ to be the signature and considering all relations in $\p$ as defining identities of the variety,
so that the variety $W$ can be recovered by $\p$ ``up to a change of variables''. Moreover, one can consider 
the algebras from $W$ as vector spaces $V$ with the actions $ \p(n): V^{\otimes
n} \to V$ compatible with compositions  and the
$\s$-module structures, so that the algebras of $W$ are recovered by $\p$ up to  isomorphisms. 

Given an $\s$-module $X$, one can define also a {\em free operad} $\f(X)$ generated by $X$ as the span of all possible 
compositions of a basis of $X$ modulo the action of symmetric groups. For example, the free operad $\f(\s\Omega)$
on the free $\s$-module $\s\Omega$ corresponds to the variety of all algebras of signature $\Omega$. 

One can define a simpler notion of {\em non-symmetric operad} as a union $P = P_1\cup P_2\cup \dots$
with the compositions $\circ_i$ as above but without actions of the symmetric groups. To distinguish them, we refer to the operads defined above as {\em symmetric}. Each symmetric operad can be considered as a non-symmetric one.
Moreover, to each non-symmetric operad $P$ one can assign a symmetric operad  $\p$ where $\p_n = S_n P_n$ is a free $S_n$ module generated by $P_n$. Then $\p$ is called a  {\em symmetrization}  of $P$. In particular, here $\dim \p_n = n! \dim P_n$.

An $n$-th codimension of a variety $W$ is just the dimension of the
respective operad component: $c_n(W) = \dim_k \p_n$ for $\p = \p_W$.
We consider both exponential and ordinary generating series for this sequence:
\begin{equation}
\label{eq::E::gen::ser}
E_{\p} (z) := \sum_{n \ge 1} \frac{\dim \p(n)}{n!} z^n , G_{\p} (z) := \sum_{n \ge 1} {\dim \p(n)} z^n .
\end{equation}
For example, if $\p$ is  a symmetrization of a non-symmetric operad $P$ then 
$E_{\p} (z) = G_P (z)$. By {\em generating series} of a symmetric operad $\p$ we mean the exponential generating function $\p(z) = E_{\p} (z)$.
In contrast, for a non-symmetric operad $P$ its generating series is defined as the ordinary generating function 
$P(z) =  G_{P} (z)$. In the case of varieties, both the  ordinary and exponential versions of the codimension series are studied. 

If the set $\Omega $ is finite then the series $\p (z)$ defines an analytic function
in a neighborhood of zero. For example, the non-symmetric operad $\mathrm{Ass}$ of associative algebras 
is the operad defined by one binary operation $m$ (multiplication) subject to the relation  $m(m(x_1,x_2),x_3)) = m(x_1,m(x_2,x_3))$ which is the associativity identity. Its $n$-th component consists of the only equivalence class of all arity $n$ compositions of $m$ with itself modulo the relation, so that 
$
\mathrm{Ass}(z) = G_{\mathrm{Ass}} (z) = \frac{z}{1-z}.
$
Its symmetrization is the symmetric operad $\ass$ generated by two operations $m(x_1,x_2)$ and $m'(x_1,x_2)=m(x_2,x_1)$
with the $S_2$ action $(12) m' = m$ 
subject to all the relations of the form $m(m(x_i,x_j),x_k)) = m(x_i,m(x_j,x_k))$. 
By the above, we have $E_{\ass}(z) = \ass (z) = \mathrm{Ass}(z)= G_{\mathrm{Ass}} (z)$, so that $\dim \ass_n = n!$.

For the reader's convenience, let us give an approximate 
translation table in a phrase-book style for the two languages of linear universal algebra. 
Here the objects at the same line are in correspondence up to a choice of  signature
while the last two strings represent equalities. 

 \centerline{
 \begin{tabular}{rcl}
  variety & --- & the category of all algebras \\
            & &      $\qquad$ over an operad \\
  subvariety & --- & quotient operad \\
  signature &  --- & set of generators \\ % of an operad
  identities  & --- & relations \\
  free algebra & --- & free operad \\
  free algebra of a variety & --- & operad \\
  $n$-th codimension & = & dimension of the $n$-th component
  \smallskip
  \\
  (ordinary or exponential) & =  &  (ordinary or exponential)  \\
 codimension series   $\phantom{ab}$ & &  $\phantom{abc}$ generating series
   \end{tabular}
  }

%The operad $\Com$ of commutative algebras

\section{General algorithmic undecidability}

\label{sec:undec}

In the next theorem, we assume that the field $\KK$ is computable.

\begin{theorem}
\label{th:undec}
Consider the set of non-symmetric quadratic operads $P$ defined by a fixed finite set $x$ of generators and some set $r$ of quadratic relations on $x$. Let ${\mathcal H}(x)$ be the set of ordinary generating functions of all such operads with various $r$. Then there is a natural $n$ such that if $|x| \ge n$ then

(i) the  set ${\mathcal H}(x)$ is infinite;

(ii) 
for some $x$  and some rational function $Q(z)$ with integral coefficients, there does not exist an algorithm which takes as an input the list $r$ such that  there is a coefficient-wise inequality $P(z)\le Q(z)$ and returns {\em TRUE} if 
the equality $P(z)= Q(z)$ holds and {\em FALSE} if not;

 (iii) for some $x$ and some rational number $q$, there does not exist an algorithm which takes as an input the list $r$ such that $c(P) \le q$ for $c(P)= \lim\sup_{n\to\infty} \sqrt[n]{\dim P_n} $  and returns {\em TRUE} if  $c(P) = q$
 and {\em FALSE} if $c(P) < q$.
%given $r$ such that the number  $c(\p)= \lim\sup_{n\to\infty} \sqrt[n]{\dim \p_n} $ is finite and   a rational number $q$
% such that $c(\p) \le q$, there does not exist an algorithm to decide whether $c(\p)= q$ or $c(\p) < q$. 
\end{theorem}

Note that Dotsenko and Brenner conjectured that `for given arities of generators $a_1,\dots , a_d$ and weights
of relations $w_1, \dots,w_r$, the set of possible Hilbert series of operads with these
types of generators and relations is always finite'~(see \cite{db}, Conjecture 10.4.1.1).  
So the part (ii) of Theorem~\ref{th:undec} solves  this conjecture negatively even in the case $a_1 = \dots =a_d=2$
and $w_1= \dots =w_r =3$. Using the family of algebras from~\cite[Example~3.1]{iyudu2017automaton}
in place of the algebra $A$ in the proof below we see that one can put here  $d=r=3$.

\begin{proof}
It is pointed out by Dotsenko~\cite{dots2016} that each graded connected associative algebra $A = A_0\oplus A_1\oplus\dots$
(where $A_0 = k$) can be considered as a non-symmetric operad $P = P(A)$ by putting $P_k = A_{k-1}$ for all $k\ge 1$,
$A_m\circ_i A_n = 0$ for $i\ge 2$ and $a \circ_1 b = ab $ for homogeneous $a,b \in A$. Moreover, the operad $P$ is quadratic if $A$ is quadratic. In this case the relations of $P$ can be easily recovered by the relations of $A$. Obviously, the ordinary generating function  of $P$ is $P(z) = zA(z)$ where $A(z) = \sum_{k\ge 0} \dim A_k$ is the Hilbert series of the algebra $A$. Moreover, in this case 
 \begin{displaymath}
c(P)= \lim\sup_{n\to\infty} \sqrt[n]{\dim P_n} = \lim\sup_{n\to\infty} \sqrt[n]{\dim A_{n-1}} 
\end{displaymath}
 \begin{displaymath}
= \lim_{n\to\infty} \sqrt[n]{\dim A_{n}} 
 \end{displaymath}
is the exponent of growth (aka entropy) $h(A)$ of the algebra $A$.  Parts (i) and~(ii) follow from the corresponding theorems on the Hilbert series of quadratic algebras by~\cite{an2}. The theorem on the exponent of growth of quadratic algebras
which is analogous to~(iii) has been proved in~\cite{gr} using Anick's construction. So,  part~(iii) follows as well. 
\end{proof}

In fact, one can show even more in part~(iii). 
It follows from the results of~\cite{an2} and~\cite{gr} that there is a quadratic polynomial $f(z) = 1-gz+rz^2$ with two rational roots $p^{-1} $ and $q^{-1} $ (where $q$ is the number from the part~(iii) and $|p| < q$) such that either $\dim P_n < \gamma c_2^n$ for all $n$ with $\gamma>0, 0<c_2<q$ or $A(z) = 1/f(z)$. 
In the last case, the sequence $\{\dim P_n\}$ is  a linear recurrence of order 2, so that $\dim P_n = \alpha q^n +\beta p^n \cong \alpha q^n$ for some $\alpha, \beta >0$. For the symmetrization $\p$ of $P$ with  $\dim \p_n =  n! \dim P_n$ we have  $\dim \p_n < \gamma c_2^n n!$ in the first case and $\dim \p_n \cong \alpha q^n n!$ in the second case. 
By part~(iii), there is no algorithm to separate these two cases. 
This means that there is no algorithm to recognize the asymptotic growth both for symmetric and non-symmetric operads.

\section{Estimates for the growth of operads}

\label{sec:est}

In this section, we discuss an approach to lower bounds for generating series of symmetric operads based on the results of~\cite{kurosh}.
Note that Dotsenko has obtained the same results and even more in the case of 
monomial shuffle operads~\cite[Section~3]{dots2012}.

The next theorem is an operadic version of famous Golod--Shafarevich
theorem which gives a criterion for an associative algebra to be
infinite-dimensional~\cite{gs}. 
%Despite the original version of the proof of
%Golod and Shafarevich~\cite{gs} can be almost directly translated
%to the language of operads (with Shafarevich complex replaced by
%the first step of the construction of a minimal model of the
%operad $\p$), we prefer to adopt another  approach (explained by
%Ufnarovski\cite{ufn}) based on a construction of the minimal free
%resolution of the trivial module.
The proof of the first statement of it is sketched in~\cite[Theorem~4.1]{kurosh}.

\begin{theorem}
\label{th:GS_oper}
 Let $\p$ be a symmetric  operad minimally generated by an
$\s$-module $X \subset \p$ with a minimal $\s$-module of relations
$R \subset \f(X)$. We assume here that both these
$\s$-modules are locally finite, that is, all their graded
components are of finite dimension.

Suppose that the formal power series  $t/f(t)$
 has non-negative coefficients, where  $f(t) = t - X(t)
            +R(t)$. Then  the
operad $\p$ is infinite and there is a coefficient-wise inequality of formal power series
$$\p(z) \ge f^{[-1]} (z),
$$ where $f^{[-1]} (z)$ is the composition power series inverse of $f(t)$.
\end{theorem}

\begin{proof}
We are working in the category of right graded modules over the operad $\p$.
Consider the trivial bimodule $I = \p / \p_+$ (where $\p_+ = \p(2)
\oplus \p(3) \oplus \dots$ is the
    maximal ideal of $\p$, as before). For the generators of the
    beginning of its minimal free resolutions, we have
    $\tor_0^\p (I,I)\cong I$, $\tor_1^\p (I,I)\cong X$ and   $\tor_2^\p (I,I)\cong
    R$, see~\cite[Sec.~3]{kurosh}. This means that the beginning of the resolution looks as
    \begin{equation}
    \label{resol_k}
        0\to \Omega^3 \to R\circ \p \stackrel{d_2}{\to} X\circ \p \to \p \to
        I \to 0,
    \end{equation}
    where $\Omega^3$ is the kernel of $d_2$.

    Obviously, the formal power series $\Omega^3(z)$ has nonnegative
    coefficients. 
    Taking the Euler characteristics of the exact
    sequence~(\ref{resol_k}), we get an equality of formal power
    series
    $$
    \Omega^3(z) =  (R\circ \p)(z) - (X\circ \p)(z) + \p(z) -I(z)
    \ge 0.
    $$
    Since $I(z)=z$, we obtain a coefficient-wise inequality
    $$
         R(\p(z)) - X(\p(z)) +\p(z) - z \ge 0,
    $$
    or
    $$
            f(\p(z))  = z+ \Omega^3(z) \ge z.
    $$
    Let $y = z+ \Omega^3(z)$. By the Lagrangian inverse formula,
    we have
    $$
           \p(z) = f^{[-1]}(y) = \sum_{i\ge 0} \frac{\pi_n}{n!}y^n,
    $$
    where $\pi_n$ is the value of the derivative
    $ \frac{d^{n-1}}{d t^{n-1}} \left( \frac{t}{f(t)} \right)^{n}
    $ at $t=0$.

    %Note that the formal power series
    Since the formal power series $t/f(t)$
     has nonnegative coefficients, it follows that the series $\left( \frac{t}{f(t)}
    \right)^{n}$ and all its derivations satisfy this property as well.
    It
    follows that  $\pi_n \ge 0$ for all $n\ge 0$. Moreover, if $a_k z^k$ is any positive summand in the
    decomposition of $t/{f(t)} $, then for every $n\ge 0$ we have
    $
    \pi_{kn+1} \ge (kn)! \left( \begin{array}{c} kn+1 \\ n \end{array}
    \right) a_k > 0,
    $
    so that the series $f^{-1}(y)$ is infinite with nonnegative coefficients.
    It follows that $\p(z) = f^{[-1]}(z +\Omega^3(z))
    \ge f^{[-1]} (z) $. In particular, the series $\p(z)$ is infinite.
\end{proof}

\begin{cor}
\label{cor:GS_binary}
Let $\p$  be a symmetric operad and let $X$ and $R$
be as above. Suppose that  $\p$ is generated by binary operations
(that is, $X=X(2)$). Suppose that the function
$$
            \phi(z) = 1 - \frac{X(z)}{z} +  \frac{R(z)}{z}
$$
is analytic in a neighbourhood of zero (it is always the case if
$X$ is finitely generated) and  has a positive real root $z_0$ in
this neighbourhood.
% such that $\phi(z_0)' \ne 0$. 
 Then the dimensions 
of the operad's components 
satisfy 
%$\frac{\dim \p_{n+1}}{(2n)!} \le \left( \frac{\dim X}{2} \right)^n.$
the inequalities
$$
\frac{(2n)!}{(n-1)!} z_0^{-n} \le  \dim \p_{n+1} \le \frac{(2n)!}{n!} (\dim X/2)^n.
$$
\end{cor}

\begin{proof}
In the above notations, we have $$
           \p(z) = f^{[-1]}(y) = \sum_{n\ge 0} \frac{\pi_n}{n!}y^n \ge \sum_{n\ge 0} \frac{\pi_n}{n!}z^n ,
    $$
    where $\pi_n$ is the value of the derivative
    $ \frac{d^{n-1}}{d t^{n-1}} \left( \phi(z)^{-n}  \right)
    $
    at $z=0$. This means that $\pi_n /(n-1)!$ is the coefficient of $z^{n-1}$ in the formal power series $ \phi(z)^{-n}$.
    
    Put $a = \dim X /2$. The 
 series $\phi(z)$
 has the form $1-az+ z^2r(z)$, where $a >0$ and the series $r(z)$ has nonnegative coefficients. Then 
it follows from~\cite{me} that there is a coefficient-wise
inequality 
$$\phi(z)^{-1} \ge \sum_{k\ge 0} z_0^{-k} z^k = (1-z/z_0)^{-1}.$$
Using the binomial theorem, we get
 $$ \phi(z)^{-n} \ge (1-z/z_0)^{-n} = \sum_{k\ge 0}  \frac{(n+k-1)!}{k!(n-1)!} z_0^{-k} z^k .$$
So, we get the first inequality:
$$
\dim \p_{n+1} \ge \pi_{n+1} \ge n! \left( \frac{(2n)!}{n!(n-1)!} z_0^{-n}\right) =  \frac{(2n)!}{(n-1)!} z_0^{-n}.
$$
On the other hand, for the free operad $F$ generated by $X$ we have $R=0$ and $\Omega^3 =0$, so that
$$
F (z) = (t-X(t))^{[-1]} = (t-at^2 )^{[-1]} = \frac{1-\sqrt{1-4a z}}{2a} 
$$
$$ 
= \sum_{n\ge 0} z^{n+1} C_n a^n, 
$$
where 
$C_n = \frac{(2n)!}{n!(n+1)!}$ is the $n$-th Catalan number. The coefficient-wise inequality 
$F(z) \ge \p(z)$ gives the second inequality $\dim \p_{n+1} \le \frac{(2n)!}{n!} (\dim X/2)^n$.
\end{proof}

Consider, in particular, the case of quadratic operad, that is, an operad having generators of arity 2 and relations of arity 3 only. 

\begin{cor}
\label{cor:GS_quadr}
Let $\p$  be a quadratic symmetric operad, that is,  $X=X(2)$ and $R=R(3)$, and let $c= \dim X$ and $d = \dim R$.
If $d\le 3 c^2/8$, then 
$$
\p(z) \ge (z-\frac{c}{2}z^2+\frac{d}{6}z^3)^{[-1]} .
$$
In particular, $\frac{(2n)!}{(n-1)!} z_1^{-n} \le  \dim \p_{n+1} \le \frac{(2n)!}{n!} (c/2)^n$,
where \linebreak
 $z_1 = 3\frac{c/2-\sqrt{c^2/4 - 2d/3}}{d}$.
\end{cor}

\begin{rema}
Note that the inequality of Corollary~\ref{cor:GS_binary} 
can be re-written in the form $c_1^n n^n \le \dim \p_n \le c_2^n n^n$
where the constants $c_1$ and $c_2$
 are constructively recovered by the generators and the relations of the operad $\p$.
In general, this result cannot be effectively re-written in the form~$\dim \p_n \sim c^n n^n$ 
because even if such $c$ exists it cannot be algorithmically recognized according to~Theorem~\ref{th:undec}. 

Note that for infinitely presented operads similar asymptotics does not generally exist even for operads generated by a single binary operation, that is, for a variety of non-associative algebras~\cite{zaitsev_no}.
\end{rema}

\section{Monomial bases and Groebner bases in operads}

\label{sec:groeb}

The Groebner bases in operads are introduced in~\cite{dk}. We also refer the reader to~\cite{Loday_Valet} and~\cite{db}
for the details. Here we briefly recall some basics. 

Fix a discrete set $\Omega$ of generators of a (symmetric or non-symmetric) free operad.
 A {\em nonsymmetric monomial} is a multiple composition of operations from $\Omega$. 
 A {\em symmetric monomial} is a nonsymmetric monomial applied to pairwise different  variables $x_{i_1}, x_{i_2}, \dots$ as a composite operation. Each monomial is represented by a rooted planar tree with 
 internal vertices labelled by operations. We assume that the edges of the tree lead from the root to the leaves which are free edges. In the case of symmetric  monomial, the leaves are also labelled by  the variables $x_{i_1}, x_{i_2}, \dots$ 

All non-symmetric monomials (including the empty monomial corresponding to the identical operation) form a linear basis of the free non-symmetric operad generated by $\Omega$. In contrast, all symmetric monomials generate the free operad $\f = \f(\s\Omega)$ as a linear dependent set. To describe a linear basis of $\f$, let us mark each edge in the tree by the minimal label of a leaf which can be reached from that edge. A symmetric monomial is called {\em shuffle monomial} if for every internal vertex and for the root, the minimal edge leading from it is the leftmost. Then the shuffle monomials labelled by the sets $\{ x_1, \dots , x_n \}$ in some ordering (where for each monomial $n$ is equal to its arity) form a linear basis of $\f$.

Two non-symmetric monomials are called isomorphic if they are isomorphic as labelled trees. Two symmetric monomials are isomorphic 
if the underlying non-symmetric monomials are isomorphic and the lists of the labels of their leaves collected from the left to the right are isomorphic as ordered sets.  A (non-)symmetric monomial $P$ is {\em divisible} by a (non-)symmetric monomial $Q$ if $Q$ is isomorphic to a submonomial of $P$ where 'submonomial' means a labelled subtree with the labels on leaves induced by the labels on the free edges. For example, see Figure~\ref{pic::shuffle_monom} (copied from~\cite{kp}) to ensure that the shuffle monomial 
$g(f(f(x_1,x_3)$, $g(x_2,f(x_4,x_9)$, $g(x_5,x_6,x_{11}))),x_7,f(x_8,x_{10}))$ is divisible by the shuffle monomial
$ f(x_1,g(x_2,x_3,x_4))$.

%\begin{wrapfigure}{c}{5.5in}
%\caption{Divisibility of shuffle monomials}
%\includegraphics{shuffle_monom1.eps}
%\label{pic::shuffle_monom}
%\end{wrapfigure}

\begin{figure*}
\includegraphics{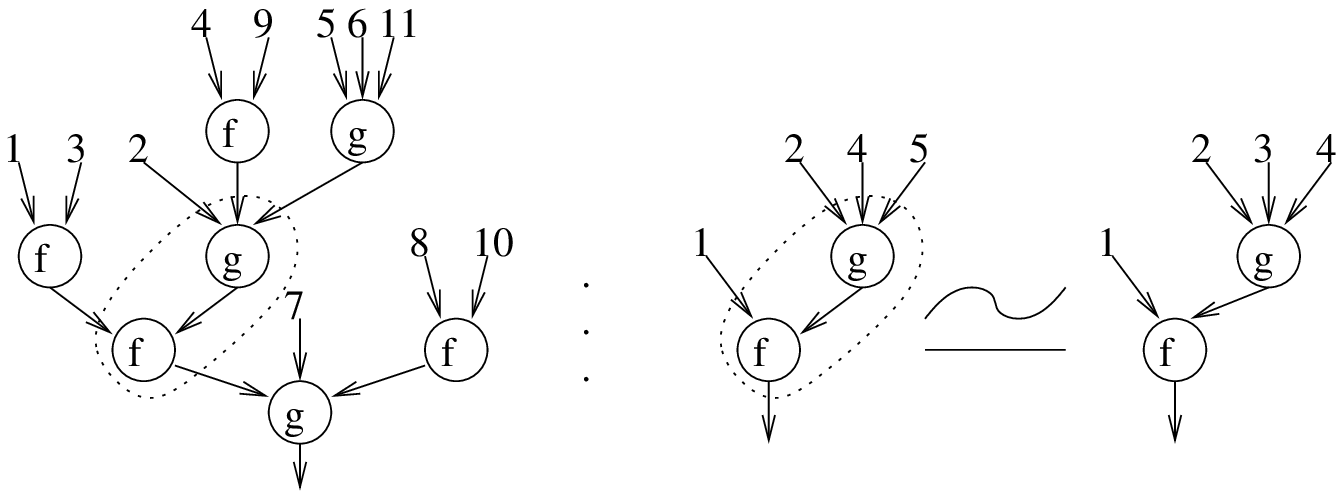}
\caption{Divisibility of shuffle monomials}
\label{pic::shuffle_monom}
\end{figure*}

A {\em shuffle composition}  is a composition of two shuffle monomials (as trees) 
whose leaves are labelled  in such a way that 
the 
composition is again a shuffle monomial and the composed monomials are isomorphic to the corresponding submonomials of the result.  The symmetric operads considered as sequences of vector spaces generated by shuffle monomials coupled with the set of all shuffle compositions are called {\em shuffle operads}. It the theory of Groebner bases, the (abstract) shuffle operads are considered in place of the symmetric operads so that the actions of the symmetric groups are  not used here.

There are families of orderings on the sets of non-symmetric and shuffle monomials which are compatible with the corresponding compositions. This defines the notion of the leading term of an element of free operad and leads to a rich Groebner bases theory. The theory includes a version of the Buchberger algorithm~\cite{dk} and even the triangle lemma~\cite{db}.
We say that an operad $\p$ has a finite Groebner basis (of relations) if the ideal of its relations admits a finite Groebner basis as an ideal of a free operad. Whereas general operad has no finite Groebner basis, a number of important operads (including the classical operads of commutative, associative and Lie algebras) admit such bases.

The only known implementation of Groebner bases algorithms for operad is the Haskell package {\sf Operads}~\cite{dv}. 
Its slightly improved version with some bugs fixed by Andrey Lando 
can be downloaded at  https://github.com/Dronte/Operads .
Experiments with operads of non-associative algebras (that is, operads generated in arity two by a two-dimensional subspace) provided by Lando show that recent version of the package allows to calculate (in a standard laptop)  the Groebner basis of an ideal generated by identities of degree 3 up to degree 6. The last degree is essentially less than the  degree 
in some analogous calculations for associative algebras provided, e.~g., by BERGMAN. One could hope that new algorithms (including a possible F4 algorithm for operads which could generalize an analogous algorithm for Groebner--Shirshov bases in associative algebras~\cite{nc_f4}) and new implementation principles will essentially extend the performance of the computer algebra software for such calculations.

\section{Growth and generating series for operads with finite Groebner bases} 

\label{sec:we}

%All theorems and methods of generating equations for generating series of operads discussed in this section are proved in~\cite{kp}.

The generating series of an operad with known Groebner basis is equal to the generating series of the corresponding {\em monomial} operad, that is, a shuffle operad or a non-symmetric operad whose relations are the leading monomials of the corresponding Groebner basis. 
The dimension of the $n$-th component of a monomial operad is equal to the number of the monomials of arity $n$ which are not divisible by the monomial relations of the operad. In this section, we consider the monomial operads only. Suppose that we know a (finite) subset $\widetilde G$ of the Groebner basis $G$ of an operad $\p$. Then we get a coefficient-wise inequality 
$$\widetilde \p(z) \le \p(z)
$$ for the monomial operad $\widetilde\p$ whose relations are the leading monomials of the elements of $\widetilde G$. This lower bound for $\p(z)$ complements the inequality of Theorem~\ref{th:GS_oper}. 

That is why the generating series of monomial operads defined by finite sets of monomial relations is of our interest here. 
For such an operad, the calculation of the dimensions of its components is a purely combinatorial problem of enumeration of the labelled trees which does not contain a subtree isomorphic to a relation as a submonomial (a pattern avoidance problem for labelled trees), see~\cite{dk-pattern}. Unfortunately, this problem is too hard to be treated recently in its full generality. In this section  we discuss some partial methods based on the results of~\cite{kp}.

First, let us discuss a simpler case of non-symmetric operads.

\begin{theorem}[\cite{kp}, Th.~{2.3.1}]
\label{th-nonsym-intro}
 The ordinary generating series of a non-symmetric operad with finite Gr\"obner
 basis is an algebraic function.
\end{theorem}

One of the methods for finding the algebraic equation for the generating series of a  non-symmetric operad $P$
 defined by a finite number of monomial relations $R$ is the following. We consider the  monomials (called stamps)  of the level  less than the maximal level of an element of $R$ which is nonzero in $P$. For each stamp $m=m_i$, we consider the generating function $y_i(z)$ of the set of all nonzero monomials which are left divisible by $m_i$ and are not left divisible by $m_t$ with $t<i$. Then the sum of all $y_i(z)$ 
is equal to $P(z)$. 
The divisibility relations on the set of all stamps leads to a system of $N$ equations of the form 
$$
y_i = f_i(z, y_1, \dots, y_N)
$$
for each $y_i = y_i(z)$, where $f_i$ is a polynomial and $N$ is the number of all stamps. Note that the degree $d_i$ of the polynomial $f_i$ does not exceed the maximal arity of generators of the operad $P$. Then the elimination of the variables leads to an algebraic equation of degree at most $d = d_1^2 \dots d_N^2$ on $P(z)$. 

A couple of similar algorithms which in some cases reduce either the number or the degrees of the equations are also discussed in~\cite{kp}.

Knowing an algebraic equation for $P(z)$, one can evaluate the asymptotics for the coefficients $\dim P_n$ 
by well-known methods~\cite[Theorem~D]{flj-slg-fn}.

Let us consider the case of shuffle operad.  A set $M$ of
shuffle monomials is called {\em shuffle regular} if for each $m\in M$
the set $M$ also contains all shuffle monomials which are obtained from $m$ by  
permutations of the labels on the leaves. Moreover, $M$ is called {\em symmetric regular} if for each other planar representation of the labelled tree $T$ defined by  $m$, all shuffle monomials that are  obtained from $T$ by  
permutations of the leaf labels also belong to $M$.

%Note that in the next theorem $\p(z) = E_\p (z)$ is the exponential generating function. 

\begin{theorem}[\cite{kp}, Cor.~{0.1.4} and Th.~3.3.2]
\label{th:main_sym_intro} Let $\p$ be symmetric
operad with a finite Gr\"obner basis, and let $M$ be the  set of leading terms of the elements of the Groebner basis. 

(a) If the set $M$ is shuffle regular, then $\p(z)$ is a differential algebraic, that is, it satisfies a non-trivial algebraic differential equations with polynomial coefficients. 

(b) If the set $M$ is symmetric regular, then $\p(z)$ is algebraic. 
\end{theorem}

The differential equation in the part~(a) is obtained by a similar way as the algebraic equation in Theorem~\ref{th-nonsym-intro}. Similar arguments lead to a system of equations of the form 
$$
y_i = C_i(z, y_1, \dots, y_N),
$$
where $C_i$ is a linear combination of the multiple compositions of the  operation $C(f,g)(z) := \int_0^z f'(w) g(w) \, dw$ where $f$ and $g$ are formal power series. The terms in the equations encode the overlapping of stamps as it is clear from the example considered below. A multiple differentiation then leads to a system of algebraic ordinary differential equations on $y_1, \dots, y_N$. The elimination of variables gives a single differential equation on $\p(z)$.

Note that in many examples, the function $\p(z)$  happens to be essentially simpler than one may expect. 
%Even an application of the corresponding algorithm for the part~(a) in the case of symmetric regular $M$ does not lead directly to an algebraic equation.  

Consider an example~\cite[Example 3.5.3]{kp}. Let $N$ be an operad with one binary operation (multiplication) 
subject to the identity $$
   [x_1,x_2][x_3,x_4] = 0,
$$
where  $[a,b] = ab-ba$ (this identity holds in the ring of the upper-triangular  matrices of order two over a non-associative commutative ring). 
The corresponding shuffle operad ${\mathcal {N}}$ is
generated by two binary generators, namely,
 the multiplication $\mu$ and
$\alpha: (x_1,x_2) \mapsto [x_1,x_2]$. Then the above identity is
equivalent to the pair of shuffle regular monomial identities
$$
f_1 = \mu(\alpha(\textrm{-},\textrm{-}), \alpha(\textrm{-},\textrm{-}) )=0
\quad \text{ and } \quad
f_2 = \alpha(\alpha(\textrm{-},\textrm{-}), \alpha(\textrm{-},\textrm{-}) )=0.
$$
Therefore, the ideal of relations of the shuffle operad ${\mathcal {N}}$
is generated by the following six shuffle monomials obtained from $f_1$ and
$f_2$ by substituting all shuffle compositions of four variables
(which we denote for simplicity by 1,2,3,4):
$$
\begin{array}{ll}
m_1 = \mu(\alpha(1,2), \alpha(3,4) ),
 & m_2 = \mu(\alpha(1,3),
\alpha(2,4) ), \\
   m_3 = \mu(\alpha(1,4), \alpha(2,3) ), &
m_4 = \alpha(\alpha(1,2), \alpha(3,4) ), \\
 m_5 = \alpha(\alpha(1,3),
\alpha(2,4) ),
  & m_6 = \alpha(\alpha(1,4), \alpha(2,3) ).\\
\end{array}
$$
The operad ${\mathcal {N}}$ is monomial, so that the monomials $m_1, \dots, m_6$
form Groebner basis for it. 
Let us describe the set $B$ of all stamps of all nonzero monomials in
${\mathcal {N}}$. Since the relations have their leaves at level 2, $B$
includes all monomials of level at most one, that is, the
monomials
$$
B_0 = \Id, B_1 = \mu(\textrm{-},\textrm{-}), B_2 =\alpha(\textrm{-},\textrm{-}).
$$
For the corresponding generating series $y_i = y_i(z)$ with
$i=0,1,2$ we have
$$
\left\{
 \begin{array}{lll} y_0 & = &z, \\
  y_1 &= & C(y_0,y_0) + C(y_1,z)+C(z,y_1) +C(y_2,z) \\
  & & +C(z,y_2) +   C(y_1,y_1)+   C(y_1,y_2)+  C(y_2,y_1), \\
  y_2 & = &C(y_0,y_0) + C(y_1,z)+C(z,y_1) +C(y_2,z) \\
  & & +C(z,y_2) +   C(y_1,y_1)+   C(y_1,y_2)+  C(y_2,y_1).
  \end{array}
 \right.
$$
Here all terms correspond to compositions of stamps, e.g., the term $C(y_1,y_2)$ in the second line denotes `the stamp of the composition $\mu(B_1,B_2)$ is $B_1$' etc. 

We see that $y_1(z) = y_2(z)$  and ${\mathcal {N}}(z) = y(z) =
y_0(z)+y_1(z)+y_2(z) = z+2y_1(z)$. The second equation of the
above system gives, after differentiation, a differential equation on $y_1$ which is equivalent 
to the equation
$$
(y'(z)-1)(2-z-3y(z)) = 4y(z)
$$
on $y(z)$ with the initial condition $y(0) = 0$.
Surprisingly, the solution of this non-linear differential equation is an algebraic function
$$
\begin{array}{l}
{\mathcal {N}}(z) = y(z) = \frac{1}{3}\left( 2-z-2 \sqrt {1-4\,z+{z}^{2}} \right) \\
= z+{z}^{2}+2\,{z}^{3}+{\frac {19}{4}}{z}^{4}+{\frac
{25}{2}}{z}^{5}+{ \frac {281}{8}}{z}^{6}+{\frac
{413}{4}}{z}^{7}+
o(z^{7})
%{\frac {20071}{64}}{z}^ {8}+{\frac{31249}{32}}{z}^{9}+{\frac {396887}{128}}{z}^{10} + o(z^{10})
 \end{array}
 $$
 
On the other hand, one can see that the set $\{f_1,f_2\}$ is symmetric regular. This leads to additional symmetries in the above integral equations. Using the identities $y_1=y_2, y_0=z$ and applying the formula $C(f,g)+C(g,f) = fg$, we get the functional equation
$$
2y_1 = z^2 + 4zy_1+3y_1^2
$$
which immediately implies the same formula for ${\mathcal {N}}(z)$.

%\begin{quest}
%Is there algorithm to decide, given an system  of ordinary differential equations in terms of the operation $C$ above, if it implies an algebraic equation for the variable $y$?
%\end{quest}

In fact, in all examples considered in~\cite{kp} the resulting functions $\p(z)$ are holonomic, that is, these functions  satisfy linear differential equations with polynomial coefficients. This means that 
their coefficient asymptotics can be calculated by recent computer algebra tools, see in particular~\cite{kauers}. 
We see that in these cases the ODE system generated by our algorithm imply a linear ODE with polynomial coefficients. 
For more complicated examples, one could 
apply the tools based on the differential algebra elimination theory to obtain a single differential or functional equation for $\p(z)$ in its simpler form.

If the growth of the operad is bounded, our equations give even more. The next theorem explains, in particular, why so simple operad as the operad Com of commutative algebras has non-algebraic exponential generating function $e^z-1$.

\begin{theorem}
\label{th:intro_slow_growth}
Let $\p$ be an operad with a finite Gr\"obner basis and let $M$ be the set of its leading terms.
Suppose that either \\
$\phantom{1111}(i)$ $\p$ is non-symmetric and the numbers $\dim \p(n)$ are bounded by some polynomial
in $n$\\
or\\
$\phantom{1111}(ii)$ $M$ is shuffle regular and the dimensions $\dim \p(n)$  are bounded by an exponential function
$a^n$ for some $a>1$.\\
Then  the ordinary generating series $G_\p (z)$ is rational.
\end{theorem}

Of course the methods for finding the asymptotics of coefficients of such rational functions are well-known for centuries. In this case the sequence of codimensions is linear recurrent.

\bibliographystyle{ACM-Reference-Format}
\bibliography{biblio_operads} 

\balance

\end{document}